\documentclass[]{article}
\usepackage{amsmath,amsthm,amssymb}
\title{Grid homology for spatial graphs and a skein exact sequence}
\author{Zipei Zhuang}

\usepackage{pifont}
\usepackage{indentfirst}
\usepackage{setspace}
\usepackage{mathrsfs}
\usepackage[backref]{hyperref}

\usepackage{subfigure}
\usepackage{graphicx}
\usepackage{import}
\usepackage{xifthen}
\usepackage{pdfpages}
\usepackage{transparent}
\newtheorem{definition}{Definition}%
\newtheorem{theorem}{Theorem}%
\newtheorem{lemma}{Lemma}%
\newtheorem{corollary}{Corollary}%
\newtheorem{proposition}{Proposition}
\begin{document}
	
	\maketitle
	
	\begin{abstract}
		We defined a grid homology theory for spatial graphs, which is slightly different from the one in  \cite{Harvey_2017}  . We showed that the skein exact sequence of singular knots in  \cite{https://doi.org/10.1112/jtopol/jtp032}    can  be extended to our grid homology for spatial graphs. 
	\end{abstract}

\section{Introduction}
  Knot Floer homology is an invariant for knots in $S^3$, defined using Heegaard diagrams and holomorphic disks, see \cite{OZSVATH200458} \cite{MR2704683}.
  In  \cite{nmsl}   , knot Floer homology was generalized to singular knots in $S^3$. An skein exact sequence was constructed in  \cite{https://doi.org/10.1112/jtopol/jtp032}   , relating the Floer homology of a knot and two resolutions at some crossing.  
  Iterating the exact sequence, they arrived at a description of the knot Floer homology groups of an arbitrary knot in terms of the knot Floer homology groups of fully singular knots, which can be explicitly calculated. This gives rise to a cube of resolution for knot Floer homology, which provides conjectural relations with Khovanov-Rozansky homology, see \cite{MR3229041} \cite{Dowlin_2018}.
  
  In this paper, we want to generalize the skein exact sequence to the case of spatial graphs. Heegaard Floer homology of spatial graphs was defined in \cite{bao2018floer} and in \cite{Harvey_2017} combinatorially using grid diagrams. We adopted  \cite{Harvey_2017}  's construction, with one difference: we represent the vertices of a spatial graph using $X$'s (see Section \ref{2} for explicit definition) , while the authors of   \cite{Harvey_2017}   used $O$'s . This is in accordance with the definition in  \cite{nmsl}   for singular knots. This does not need much additional work: most arguments in \cite{Harvey_2017} can be used in our definition, with little modification. There are also some differences: e.g. compare Prop.\ref{1}   and Proposition 4.21 in  \cite{Harvey_2017}.
  
  In Section 1, we reviewed the grid diagram representation of a spatial graph; In Section 2, we constructed the grid homology of a spatial graph and proved some basic properties; In Section 3, we developed a skein exact sequence for a spatial graph at a vertice.

 \section{Grid diagrams of a spatial graph}\label{2}
 
 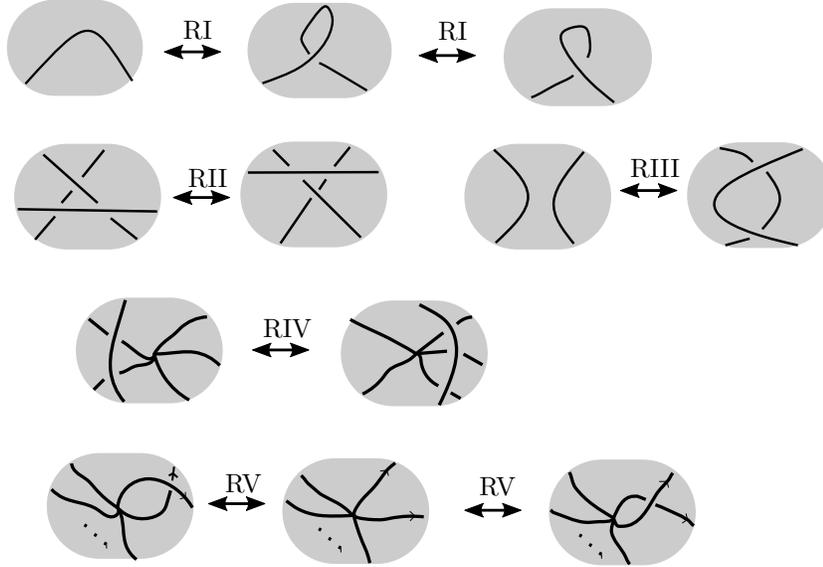
\begin{figure}[t]
 	\def\svgwidth{\columnwidth}
\begingroup%
  \makeatletter%
  \providecommand\color[2][]{%
    \errmessage{(Inkscape) Color is used for the text in Inkscape, but the package 'color.sty' is not loaded}%
    \renewcommand\color[2][]{}%
  }%
  \providecommand\transparent[1]{%
    \errmessage{(Inkscape) Transparency is used (non-zero) for the text in Inkscape, but the package 'transparent.sty' is not loaded}%
    \renewcommand\transparent[1]{}%
  }%
  \providecommand\rotatebox[2]{#2}%
  \newcommand*\fsize{\dimexpr\f@size pt\relax}%
  \newcommand*\lineheight[1]{\fontsize{\fsize}{#1\fsize}\selectfont}%
  \ifx\svgwidth\undefined%
    \setlength{\unitlength}{209.76377953bp}%
    \ifx\svgscale\undefined%
      \relax%
    \else%
      \setlength{\unitlength}{\unitlength * \real{\svgscale}}%
    \fi%
  \else%
    \setlength{\unitlength}{\svgwidth}%
  \fi%
  \global\let\svgwidth\undefined%
  \global\let\svgscale\undefined%
  \makeatother%
  \begin{picture}(1,0.7027027)%
    \lineheight{1}%
    \setlength\tabcolsep{0pt}%
    \put(0,0){\includegraphics[width=\unitlength,page=1]{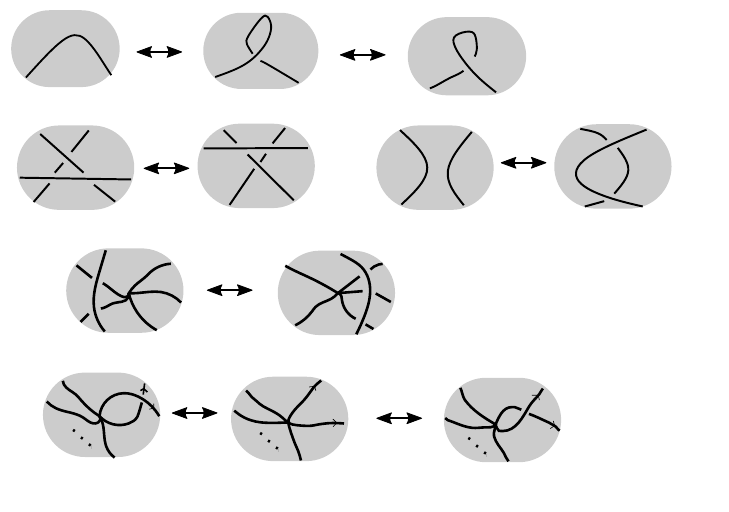}}%
    \put(0.20857859,0.64581763){\color[rgb]{0,0,0}\makebox(0,0)[lt]{\lineheight{1.25}\smash{\begin{tabular}[t]{l}RI\end{tabular}}}}%
    \put(0.48836347,0.64437601){\color[rgb]{0,0,0}\makebox(0,0)[lt]{\lineheight{1.25}\smash{\begin{tabular}[t]{l}RI\end{tabular}}}}%
    \put(0.21399622,0.48328887){\color[rgb]{0,0,0}\makebox(0,0)[lt]{\lineheight{1.25}\smash{\begin{tabular}[t]{l}RII\end{tabular}}}}%
    \put(0.69887367,0.49412412){\color[rgb]{0,0,0}\makebox(0,0)[lt]{\lineheight{1.25}\smash{\begin{tabular}[t]{l}RIII\end{tabular}}}}%
    \put(0.29616351,0.31805129){\color[rgb]{0,0,0}\makebox(0,0)[lt]{\lineheight{1.25}\smash{\begin{tabular}[t]{l}RIV\end{tabular}}}}%
    \put(0.25462839,0.14829904){\color[rgb]{0,0,0}\makebox(0,0)[lt]{\lineheight{1.25}\smash{\begin{tabular}[t]{l}RV\end{tabular}}}}%
    \put(0.5333453,0.14505149){\color[rgb]{0,0,0}\makebox(0,0)[lt]{\lineheight{1.25}\smash{\begin{tabular}[t]{l}RV\end{tabular}}}}%
  \end{picture}%
\endgroup%

 	\caption{The graph Reidemeister moves} 
 	\label{graph1}
 \end{figure}

  A spatial grgaph is an embedding of an oriented graph in $S^3$. As in the case of knots and links, we can represent a spatial graph by its diagram on the plane. Two spatial graphs are \textbf{equivalent} if they are connected by a finite number of graph Reidemeister moves (see Figure )  and planar isotopies.
  
  The first three graph Reidemeister moves are just the usual Reidemeister moves for knots and links. Note that the \textbf{RV} moves can be operated only for edges with the same orientation(both incoming or outgoing).
  
  In  \cite{Harvey_2017}  , a spatial graph (under this equivalence relation) is called a transverse spatial graph, indicating that at each vertex $v$, the incoming and outgoing edges are separated by a small disk, which intersect the spatial graph at $v$. The ambient isotopies between two transverse spatial graphs must preserve the disks.   For brevity, we omit the terminology "transverse": any spatial graph in this paper is viewed as a transverse spatial graph.
  
  A (planar) \textbf{grid diagram} is an $n \times n$ grid. Each square is empty, decorated with an $O$ or $X$ such that there is \textbf{exactly} 1 $X$ in each row or column.
  
  If we identify the top boundary segment with the bottom one, and the left boundary segment with the right one, we get a \textbf{toroidal grid diagram}. By definition, two planar grid diagrmas give rise to the same toroidal grid diagram if and only if they are related by cyclic permutations, i.e. cyclicly permute the order of their rows or columns.
  
  A grid diagram $\mathbb{G}$ specifies a spatial graph: Draw oriented segments connecting the $X$-marked squares to the $O$-marked squares in each column; then draw oriented segments connecting the $O$-marked squares to the $X$-marked squares in each row, with the convention that the vertical segments always cross above the horizontal ones. Note that under this rule, a vertex with valance $>2$ of the spatial graph must correspond to an $X$ in the grid diagram. We will call these $X$'s the \textbf{vertex X}'s, and the other \textbf{standard X}.
  
\begin{figure}
	
	\includegraphics[width=\linewidth]{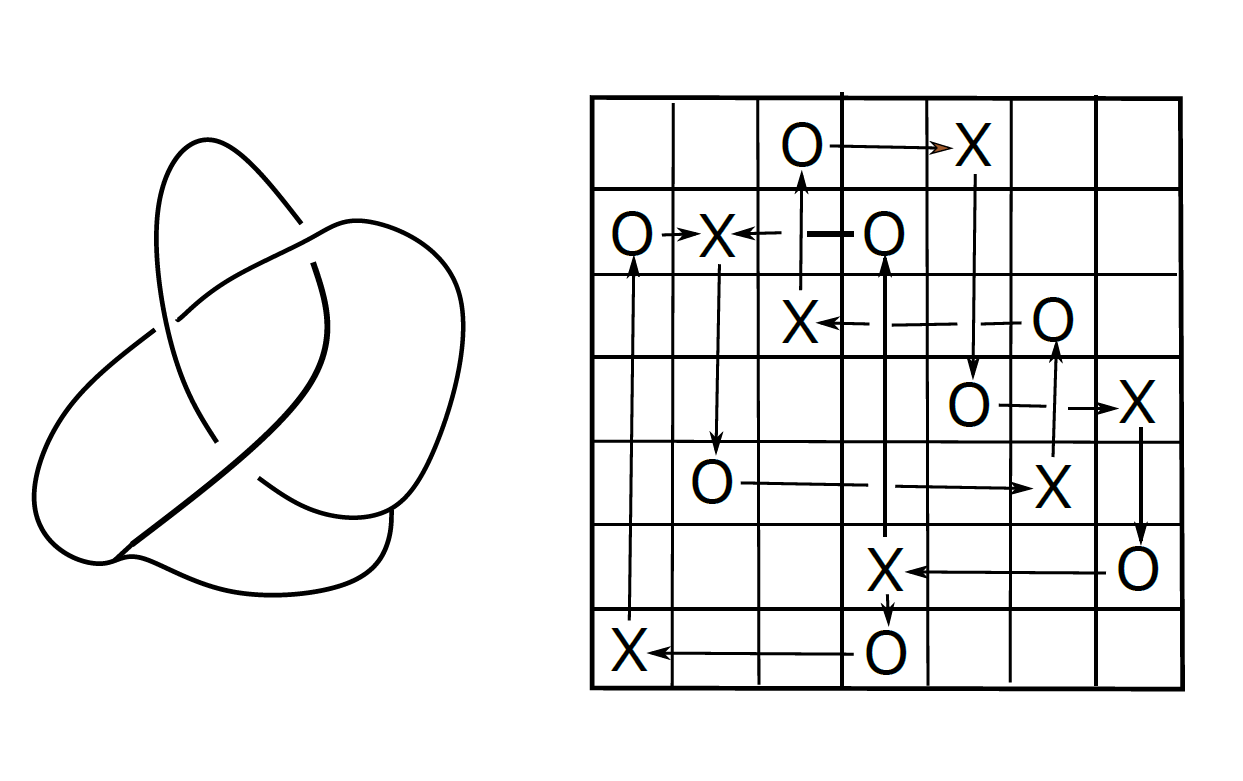}
	\caption{A spatial graph and a grid diagram for it} 
	\label{figure2}
\end{figure}

  The following properties has been proved in \cite{Harvey_2017}
  (Note that we use $X$-markings to represent vertices of the spatial graph, while in \cite{Harvey_2017}  they use $O$-markings. Of course this difference does not affect the validity of the proof ):
  
  \begin{theorem}
  	(1) Any spatial graph can be represented by a grid diagram.
  	
  	(2)If $g$ and $g'$ are two grid diagrams representing the same spatial graph, then $g$ and $g'$ are related by a finite sequence of graph grid moves.
  \end{theorem}
  
  For an vertex $X$, let the set of $O$'s that appear in a row or column with this $X$ be its \textbf{flock}. A flock is in \textbf{L-formation}, if the $O$'s are all to the right and below of the $X$(See Figure    ). A \textbf{prefered} grid diagram is one that all vertex $X$'s have their flocks in L-formation. 
  
  We need a further lemma in Section  \ref{3}    , which is proved in \cite{Harvey_2017}:

  \begin{lemma} \label{lemma1}
  	Any spatial graph can be represented by a preferred grid diagram.
  \end{lemma}

  \section{Grid homology for spatial graphs}
  \subsection{The grid chain complex}
   From now on we always assume our graph  has no vertices with only incoming edges or only outgoing edges. Equvalently, the grid diagrams representing these graphs have at least 1 $O$ in each row and column.
   
   Let $G$ be a $n \times n$ grid diagram representing a spatial graph $G$. A \textbf{grid state} for $\mathbb{G}$ is an $n$-tuple of points $\textbf{x}=\{x_1,...x_n\}$ in the torus, with the property that each horizontal circle and each vertical circle contains exactly 1 of the elements of $\textbf{x}$. The set of grid states for $\mathbb{G}$ is denoted $\textbf{S}(\mathbb{G})$. For $\textbf{x},\textbf{y} \in \textbf{S}(\mathbb{G}) $, let Rect(\textbf{x},\textbf{y}) denote the set of rectangles from $\textbf{x}$ to $\textbf{y}$. A rectangle $r \in Rect(\textbf{x},\textbf{y})$ is called an \textbf{empty rectangle} if $\textbf{x} \cap Int(r) =\textbf{y} \cap Int(r) = \emptyset $. The set of empty rectangles from $\textbf{x}$ to $\textbf{y}$ is denoted $Rect ^\circ(\textbf{x},\textbf{y})$.
   
   Suppose that $\mathbb{G}$ has m $O$ markings, denoted $O_1,O_2,...O_m$, and
   
   Denote the $O$ markings of $\mathbb{G}$ by $O_1,O_2,...O_m$, assign a variable $V_i$ to each $O_i$. Write $\mathcal{R}= \mathbb{F}[V_1,...,V_m]$, where $\mathbb{F}=\mathbb{Z}/2\mathbb{Z}$. Define
   \begin{equation}
   	C^-(\mathbb{G})= \text{the free } \mathbb{F}[U_1,...,U_m]- \text{module generated by } \textbf{S}(\mathbb{G})
   \end{equation}
and $\partial ^- :C^-(\mathbb{G}) \longrightarrow C^-(\mathbb{G})$ is the $\mathcal{R}$-module homomorphism defined by 
\begin{equation}
	\partial ^-(\textbf{x}) = \sum_{\textbf{y} \in \textbf{S}(\mathbb{G})}  \sum_{\substack{r \in Rect^\circ(\textbf{x}, \textbf{y})  \\Int(r) \cap \mathbb{X}= \emptyset}} V_1^{O_1(r)}\cdots V_n^{O_n(r)} \cdot \textbf{y}
\end{equation}
   The relative Maslov grading $M:\textbf{S}(\mathbb{G}) \longrightarrow \mathbb{Z}$ is defined by
   \begin{equation}
   	M(\textbf{x})-M(\textbf{y})= 1-2\#(r \cap \mathbb{O}) +2\#(\text{x} \cap Int(r))
   \end{equation}
for any $r \in $ $Rect(\textbf{x},\textbf{y})$.
   The relative Alexander grading $A:\textbf{S}(\mathbb{G}) \longrightarrow \mathbb{Z}$
   is defined by
   \begin{equation}
   	A(\textbf{x})- A(\textbf{y}) =\#(r \cap \mathbb{X}) -\# (r \cap \mathbb{O})
   \end{equation}
for any $r \in $ $Rect(\textbf{x},\textbf{y})$.
The gradings can be extended to $C^-(\mathbb{G})$ by defining 	M($V_i$)=$-2$, A($V_i$)=$-1$ for any $i$.

As in the case of knots, the homomorphism $\partial^-$ is a differential of degree $(-1,0)$, i.e. it decreases the Maslov grading by 1 and preserves the Alexander grading.

\begin{proposition} \label{1}
	 Let $\mathbb{G}$ be a grid diagram for the spatial graph $G$.
	
	(1)If $O_i, O_j$ lie on the interior of the same edge, the multiplication by $V_i$ is chain homotopic to multiplication by $V_j$. More generally, this is true when $O_i$ and $O_j$ can be connected by an arc on the spatial graph which does not intersect any vertex with valance $>2$.
	
	(2) Let $v$ be a vertex of $G$, $e_{i_1},...,e_{i_s}$ the incoming edges at $v$, and $e_{j_1},...,e_{j_t}$ the outgoing edges at $v$. Choose an $O_k$ on each $e_k$.   Then multiplication by $U_{i_1}\cdots U_{i_s}$ is chain homotopic to multiplication by $U_{j_1}\cdots U_{j_t}$.
\end{proposition}

\begin{figure}[t]
	\def\svgwidth{\columnwidth}
	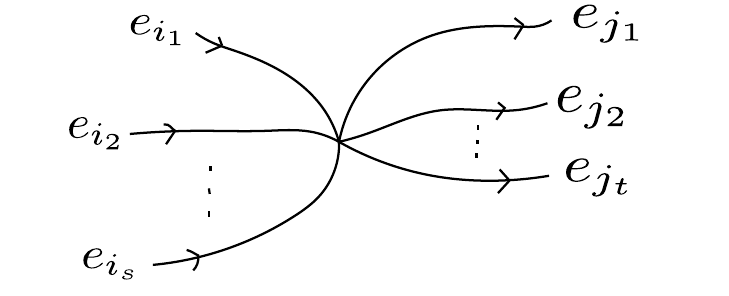 
	\label{graph3}
\end{figure}

\begin{proof}
	Choose an $X_1 \in \mathbb{X}$ on the grid diagram. Let $O_{i_1},...,O_{i_s}$ be the $O$'s that lie on the same line of $X_1$, and $O_{j_1},...O_{j_t}$ the $O$'s that lie on the same column of $X_1$.
	
	Define 
	\begin{equation}
		\mathcal{H}_1(\textbf{x})= \sum_{\textbf{y} \in \textbf{S}(\mathbb{G})}  \sum_{\substack{r \in Rect^\circ(\textbf{x},\textbf{y}) \\ Int(r) \cap \mathbb{X} =X_1}}  V_1^{O_1(r)} \cdots V_m^{O_m(r)}\cdot \textbf{y}
	\end{equation}
	
	For grid states \textbf{x} and \textbf{z}, $\psi \in \pi(\textbf{x, \textbf{z}})$, let $N(\psi)$ be the number of ways we can decompose $\psi$ as a composite of two empty rectangles $r_1 * r_2$. Then 
	\begin{equation}\label{5}
	   \partial^- \circ \mathcal{H}_1 +\mathcal{H}_1 \circ \partial ^- = \sum_{\textbf{z} \in \textbf{S}(\mathbb{G})}  \sum_{\substack{\psi \in \pi(\textbf{x},\textbf{z}) \\ \psi \cap \mathbb{X} =X_1}}  V_1^{O_1(r)} \cdots V_m^{O_m(r)}\cdot \textbf{z}
	\end{equation}
	
	As in the case for knots, when $ \textbf{x} \backslash \textbf{x} \cap \textbf{z}$ consists of 4 or 3 elements, $N(\psi)=2$, hence they contribute nothing to the right of Equation  \ref{5}  since we wwork in mod2 coefficient.
	
	When $\textbf{x}= \textbf{z}, \psi = r_1 * r_2$, in this case $r_1$ and $r_2$intersect along 2 edges and therefore $\psi$ is an annulus. Recall that our grid diagram contains an $X$ in each row and column, therefore $\psi$ is an annulus with height or width equal to 1, which contains $X_1$. 
	
	If $X_1$ is the row, then its contribution is multiplication by $V_{i_1}\cdots V_{i_s}$; If $X_1$ is the column, then its contribution is multiplication by $V_{j_1}\cdots V_{j_t}$. So we have
	\begin{equation}
		\partial^- \circ \mathcal{H}_1 + \mathcal{H}_1 \circ 	\partial^- =V_{i_1}\cdots V_{i_s}-V_{j_1}\cdots V_{j_t}
	\end{equation}
	
	If $X_1$ lies on the interior of an edge, or representing a vertex with exactly one incoming edge and one outgoing edge, then there is exactly 1 $O_i$ in the same row as $X_1$ and 1 $O_j$ in the column with $X_1$. The argument above shows that multiplication by $O_i$ is chain homotopic to multiplication by $O_j$. Iterating this procedure proves (1). (2) follows immediately from (2) and the above argument.
	
\end{proof}

	Suppose that G has n edges, and $\mathbb{G}$ has m $O$'s. Order the $O$'s so that the first n $O$'s lie on different edges. By Proposition   \ref{1}    the multiplication by any $V_i$ is chain homotopic to some $V_j, j \leq m$. In particular, they induce the same action on the homology group.

\begin{definition}
	The homology of $GC^-(\mathbb{G})$ is called the (unblocked) grid homology of $\mathbb{G}$, viewed as an $\mathbb{F}[U_1,...U_m]$-module.
\end{definition}

Define the simply blocked grid chain complex to be
\begin{equation}
	\widehat{GC}(\mathbb{G})= GC^-(\mathbb{G})/V_1=\cdots =V_n=0
\end{equation}
and the fully blocked grid complex
\begin{equation}
	\widetilde{GC}(\mathbb{G})/V_1=\cdots = V_m=0
\end{equation}
The simply blocked grid homology and fully blocked grid homology are respectively the homology of these complexes. They are $\mathbb{F}$-vector spaces.

Now we show that $\widetilde{GC}(\mathbb{G})$ is independent of the choice of the $O_1,...,O_n$.

Denote by W the 2-dimensional $\mathbb{F}$-space, with one generator in bigrading $(0,0)$, and another generator in bigrading $(-1,-1)$.
\begin{lemma}
	Let G be a spatial diagram with m edges, and $\mathbb{G}$ a grid diagram of G with m $O$'s. Then there is an isomorphism
	\begin{equation}
		\widetilde{GC}(\mathbb{G}) \cong \widehat{GH}(\mathbb{G}) \otimes W^{\otimes(m-n)}
	\end{equation} 
\end{lemma}

\begin{proof}
	We prove that 
	\begin{equation}
		H \left(\frac{GC^-(\mathbb{G})}{V_1=\cdots V_{n+k}=0}\right)   =H\left(\frac{GC^-(\mathbb{G})}{V_1=\cdots V_{n}=0}\right) \otimes W^k
	\end{equation}
inductively for k $\ge 0$.

For any $j>n$, $O_j$ lies on some edge $e_i$, then by  Prop.\ref{1}    multiplication by $V_j$ is chain homotopic to multiplication by $V_i$, so
\begin{equation}
	V_j: \frac{GC^-(\mathbb{G})}{V_1=\cdots V_{j-1}=0} \longrightarrow  \frac{GC^-(\mathbb{G})}{V_1=\cdots V_{j-1}=0}
\end{equation} 
is chain homotopic to 0 map. Hence the long exact sequence induced by
\begin{equation}
	0 \longrightarrow  \frac{GC^-(\mathbb{G})}{V_1=\cdots V_{j-1}=0} \stackrel{V_j}{\longrightarrow}  \frac{GC^-(\mathbb{G})}{V_1=\cdots V_{j-1}=0} \longrightarrow  \frac{GC^-(\mathbb{G})}{V_1=\cdots V_{j}=0} \longrightarrow 0
\end{equation}
becomes into a short exact sequence
\begin{equation}
	0 \longrightarrow H \left(\frac{GC^-(\mathbb{G})}{V_1=\cdots V_{j-1}=0}\right) \longrightarrow  H \left(\frac{GC^-(\mathbb{G})}{V_1=\cdots V_{j}=0}\right) \longrightarrow H \left(\frac{GC^-(\mathbb{G})}{V_1=\cdots V_{j-1}=0}\right) \longrightarrow 0
\end{equation}
The second arrow preserves the bigrading, and the third is homogeneous of bidegree $(1,1)$. Therefore we have
\begin{equation}
	H \left(\frac{GC^-(\mathbb{G})}{V_1=\cdots V_{j}=0}\right) = H \left(\frac{GC^-(\mathbb{G})}{V_1=\cdots V_{j-1}=0}\right) \otimes W
\end{equation}
 Iterating this process proves the proposition.

\end{proof}

\begin{corollary}
	$GC^-(\mathbb{G})$ is a finite dimensional $\mathbb{F}$-space. It is independent of the choice of the $V_i, i=1,...,n$ in the definition.
\end{corollary}

To show that $\widehat{GH(\mathbb{G})}, GH^-(\mathbb{G})$ are invariants of the spatial graph $G$, we have to prove that they are invariant under grid moves. The proof, however, has no essential difference from that in  \cite{Harvey_2017}. The  only drawback of our homology groups is that the usual combinatorial definition for Maslov grading cannot be generalized to this case, so we confined ourselves with relative Maslov grading, which is easily shown to be invariant under grid moves.

Finally, the homology groups are invariant under orientation-reversing:
\begin{proposition}
	The relatively graded $\mathbb{F}$ vector space $\widehat{GH(\mathbb{G})}$ and the relatively graded $\mathbb{F}[V_1,...V_n]$-module $GH^-(\mathbb{G})$ are invariant if we change the orientation of each edge of $G$.
	\end{proposition}

\begin{proof}
	Denote by $G'$ the spatial graph obtained by reversing all the orientations on edges of $G$. Let $\mathbb{G}$ be a grid diagram of $G$. Reflecting $\mathbb{G}$ along some diagonal, we get a grid diagram $\mathbb{G}'$ for $G'$. The reflection gives rise to a bijection from $\textbf{S}(\mathbb{G})$ to $\textbf{S}(\mathbb{G}')$, and a bijection of rectangles that appear in the differential. The relative Maslov grading and Alexander grading depends only on the position of the $O$ and $X$'s, so is also preserved by the reflection.
\end{proof}

\section{The skein exact sequence} \label{3}

Let $G$ be a spatial graph. By Lemma  \ref{lemma1}  ,there exists a preferred $m \times m$ grid diagram $\mathbb{G}$ for $G$. Suppose that $v$ is a vertex of $G$, and denote by $I(J)$ the set of incoming(outgoing) edges at $v$. Require that $|I| \ge 2, |J| \ge 2$. We separate $I,J$ into two disjoint nonempty sets: $I=A \cup B, J= C \cup D$. Write $A= \{a_1,...,a_i\}$, $B=\{b_1,...b_j\}$, $C=\{c_1,...c_k\}$, $D=\{d_1,...,d_l\}.$

We define $G_+, G_-, \mathcal{R}_G$ by modify $G$ locally at $v$, see Figure  \ref{graph4}  . For convenience, we temporarily call $G_+(G_-)$ and $\mathcal{R}_G$ the two resolutions of $G$ at $v$. Note that this differs from the usual meaning of "resolution" at a crossing.

Suppose that the vertex $v$ corresponds to $X_v \in \mathbb{X}$ of $\mathbb{G}$, and let $O_{a_1},...,O_{a_i}$ be the $O$'s that is adjacent to $X_v$ and lies on the edge $a_1,...a_i$, respectively. Write $O_a=\{O_{a_1},...,O_{a_i}\}$, $V_a=V_1 \cdots V_i$. Similarly define $O_b,O_c, O_d$, and $V_b, V_c, V_d$.

The main result of this section is the following:

\begin{theorem}\label{thm1}
	There are long exact sequences:
	\begin{equation}\label{a}
		\cdots \longrightarrow GH^-(G_+) \longrightarrow H_*\left(\frac{GC^-(G)}{V_a +V_b -V_c -V_d}\right) \longrightarrow GH^-(\mathcal{R}_G) \longrightarrow GH^-(G_+) \longrightarrow \cdots
	\end{equation}
	
		\begin{equation}\label{b}
		\cdots \longrightarrow \widehat{GH}(G_+) \longrightarrow H_*\left(\frac{\widehat{GC}(G)}{V_a +V_b -V_c -V_d}\right) \longrightarrow \widehat{GH}(\mathcal{R}_G) \longrightarrow \widehat{GH}(G_+) \longrightarrow \cdots
	    \end{equation}
	
	\begin{equation}\label{c}
		\cdots \longrightarrow GH^-(G_-)  \longrightarrow GH^-(\mathcal{R}_G) \longrightarrow H_*\left(\frac{GC^-(G)}{V_a +V_b -V_c -V_d}\right)
		\longrightarrow GH^-(G_-) \longrightarrow \cdots
	\end{equation}

		\begin{equation}\label{d}
		\cdots \longrightarrow \widehat{GH}(G_-)  \longrightarrow \widehat{GH}(\mathcal{R}_G) 
		\longrightarrow H_*\left(\frac{\widehat{GC}(G)}{U_a +U_b -U_c -U_d}\right)
		\longrightarrow \widehat{GH}(G_-) \longrightarrow \cdots
	\end{equation}

\end{theorem}

\begin{figure}[t]
	\def\svgwidth{\columnwidth}
	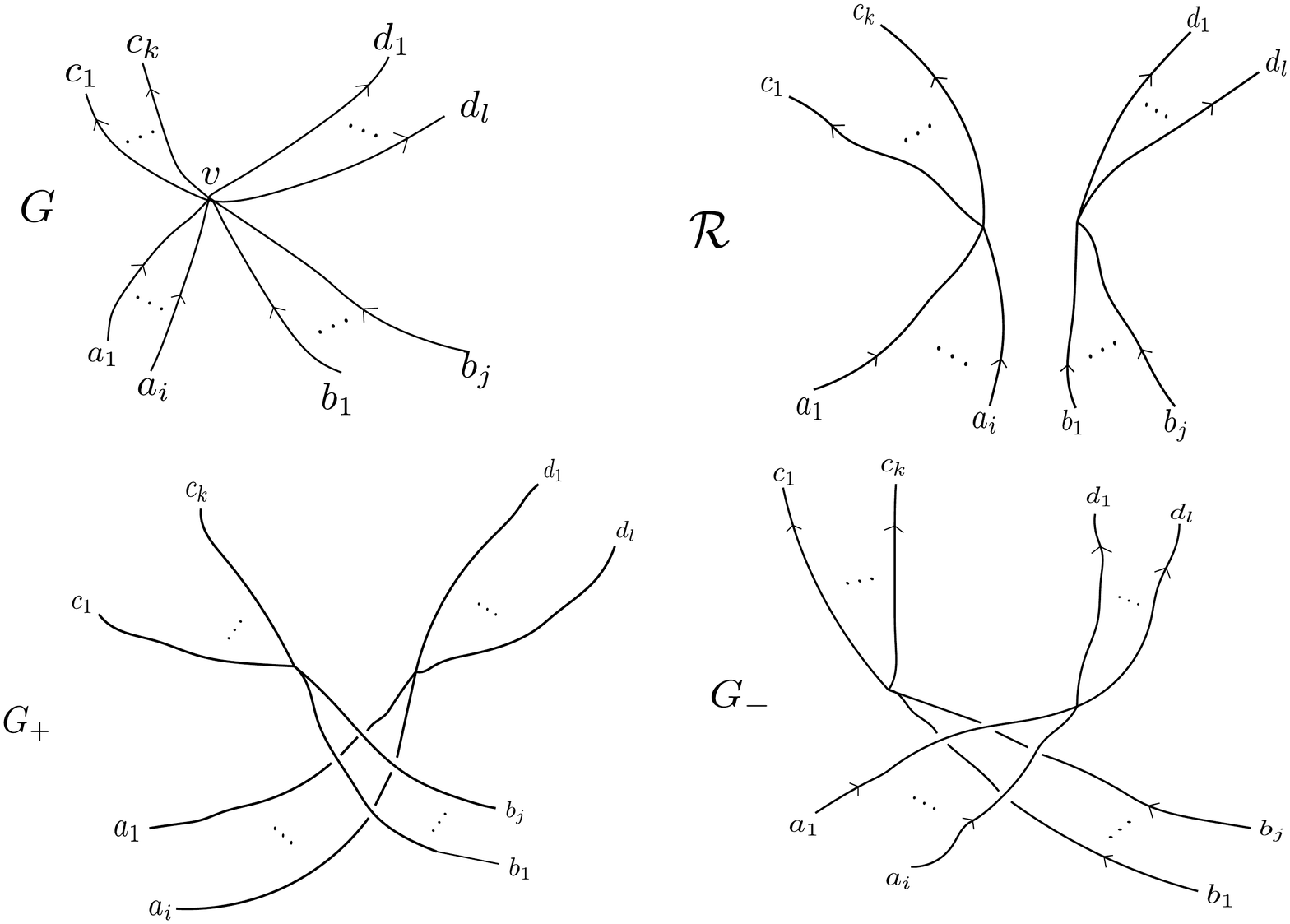
	\caption{The resolutions at a vertex} 
	\label{graph4}
\end{figure}

\textbf{Remark}. When $G_+(G_-)$ is a knot/link, the above proposition is just Theorem 4.1 of \cite{https://doi.org/10.1112/jtopol/jtp032}.

Since $\mathbb{G}$ is a preferred grid diagram, the $O_a, O_b$ are below $A$ and $B$, and $O_c, O_d$ are to the right of $A,B$. Define $R_{X_v}(C_{X_v})$ to be the row(column) which X lies on. Ignore $X$ first, split $R_X$ into two rows, so that $O_a$ and $O_b$ are separated. Similarly, split $C_X$ into two columns, so that $O_c$ and $O_d$ are separated(See Fig.\ref{graph5})

Let $\mathbb{X}_0$ be the $m-1$ $X$'s in the complement. The original square of $X_v$ becomes a $2 \times 2$ grid which lie at the intersection of the two new rows and the two new columns. Mark the upper -left and lower-right squares by $A$, and mark the upper-right and lower-leftsquares by $B$. We call these pairs of squares $\mathbb{A}$ and $\mathbb{B}$ respectively. If we add an $\mathbb{X}$ on each square of $\mathbb{A}$, we get a grid diagram $\mathbb{G}_A$ for the spatial graph $G_+$; If we add an $X$ on each square of $\mathbb{B}$, we get a grid diagram $\mathbb{G}_B$ for the spatial graph $\mathcal{R}$.

Denote by $p$ the intersection point of the two $A$'s and the two $B$'s. Consider the submodule $Z$ of $GC^-(\mathbb{G}_B)$ which is generated by all states containing the corner $p$. It is a subcomplex of $GC^-(\mathbb{G}_B)$:  For any $\textbf{x} \in Z, \textbf{y}\in GC^-(\mathbb{G}_B), \textbf{y}\notin Z$, an $r \in Rect(\textbf{x},\textbf{y})$ must contain some $B \in \mathbb{B}$, which is not allowed in $GC^-(\mathbb{G}_B)$, since the two $B$'s have been decorated by $X$.

\begin{figure}[t]
	\def\svgwidth{\columnwidth}
	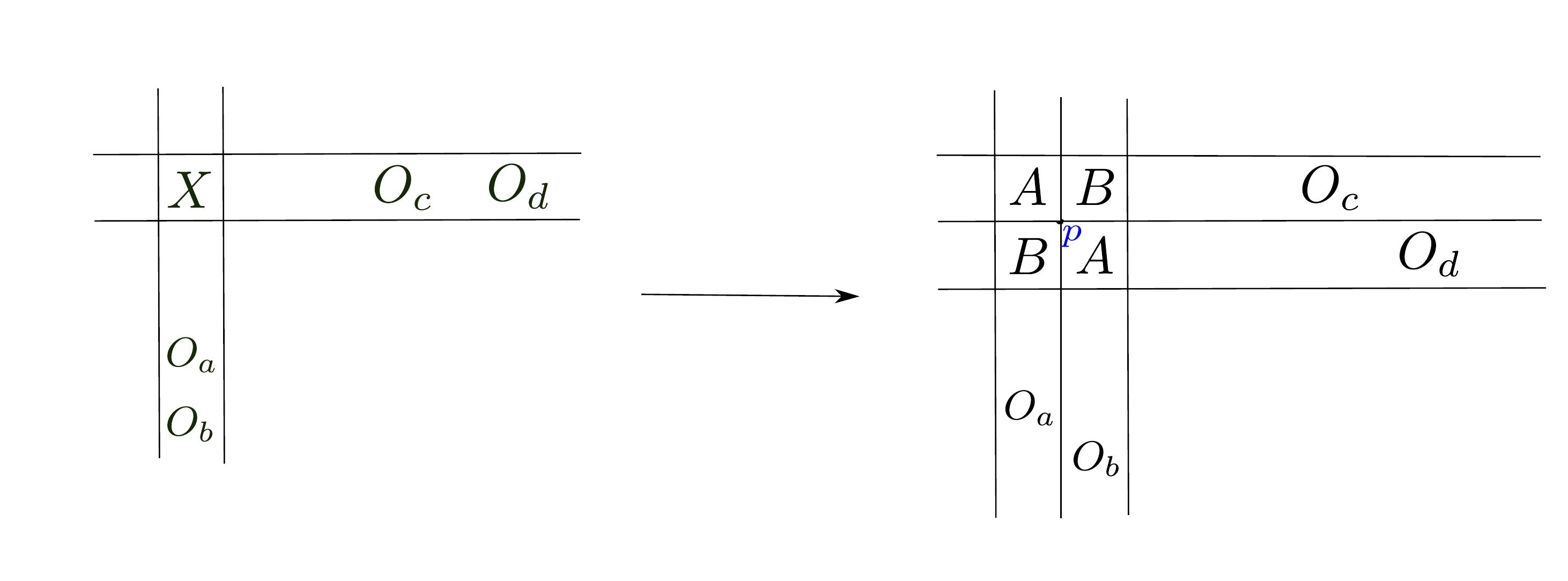
	\caption{Split the column(row) of $X_v$ into two} 
	\label{graph5}
\end{figure}

\begin{lemma} \label{lemma3}
	$Z$ is chain homotopic to $GC^-(\mathbb{G})$.
\end{lemma}
\begin{proof}
	Construct $P: GC^-(\mathbb{G}) \longrightarrow Z$ as follows.
	
	For $x \in S(\mathbb{G})$, let $x \cup \{p\}$ be the state of $S(\mathbb{G}_B)$ obtained by adding $p$ to $\textbf{x}$. This is a bijection from $S(\mathbb{G})$ to the subset of  $S(\mathbb{G}_B)$ containing $p$. Let $\textbf{x},\textbf{y} \in \textbf{S}(\mathbb{G})$, $r \in Rect^\circ(\textbf{x},\textbf{y})$, define $P(r) \in Rect^ \circ(\textbf{x} \cup\{p\},\textbf{y} \cup \{p\})$ as follows: If $r$ does not intersect with the interior of $C_{X_v}$ and $R_{X_v}$, then $r$ naturally induces a rectangle $P(r)$ from $\textbf{x} \cup \{p\}$ to $\textbf{y} \cup \{p\}$; If $r$ intersects with the interior of $R_{X_v}$, the intersection is a $1 \times k$ rectangle, which splits into a $2 \times k$ rectangle in $\mathbb{G}_B$. Define $P(r)$ to be the rectangle containing this $2 \times k$ rectangle. The definition is similar when $r$ intersects with the interior of $C_{X_v}$(See Figure \ref{graph6} ). It is easy to see $r \longrightarrow P(r)$ gives a bijection from $Rect^\circ(\textbf{x},\textbf{y})$ to $Rect^\circ(\textbf{x} \cup \{p\},\textbf{y}\cup \{p\})$, and
	\begin{equation}
		\begin{aligned}
				\partial^- \circ P(\textbf{x})&=\partial^-(\textbf{x} \cup\{p\})  \\
				&= \sum_{\textbf{y} \in \textbf{S}(\mathbb{G}_B)} \sum_{\substack{r \in Rect^\circ(\textbf{x} \cup \{p\},\textbf{y}\cup \{p\}) \\ r \cap \mathbb{X}_0 = \emptyset}}  V_1^{O_1(r)} \cdots V_{m+1}^{O_{m+1}(r)} \textbf{y} \cup \{p\}  	\\
				&=P \circ \partial^-(\textbf{x})
		\end{aligned}
		\end{equation}
	
This shows that $P$ is a chain complex isomorphism, which obviously preserves the ralative Maslov and Alexander grading.
	
\end{proof}

	\begin{figure}[t]
		\def\svgwidth{\columnwidth}
		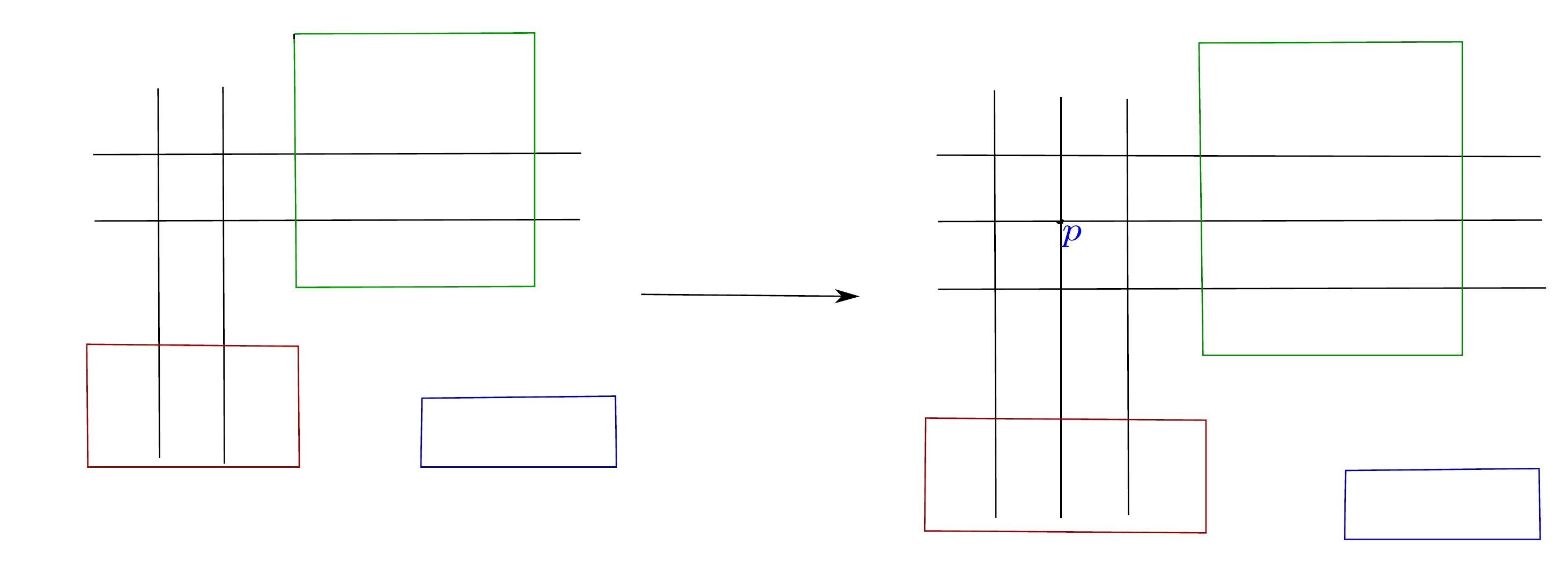
		\caption{The correspondence of the rectangles in $GC^-(\mathbb{G})$ and  $Z$}
		\label{graph6}
	\end{figure}

	Let $Y$ be the quotient complex $GC^-(\mathbb{G}_B)/Z$. As a module, $Y$ is generated by states not containing $p$, and the differential $\partial_Y$ is defined by counting rectangles that do not contain elements in $\mathbb{X} \cup \mathbb{A} \cup \mathbb{B}$. This shows that $Y$ is a subcomplex of $GC^-(\mathbb{G}_A)$. We claim that the quotient complex
	\begin{equation}
		GC^-(\mathbb{G_A}/Y) \cong Z
	\end{equation}
	
	As a module, $	GC^-(\mathbb{G_A})/Y$ is generated by all states that contain $p$, and the differential is defined by counting rectangles that do not contain elemnts in $\mathbb{X} \cup \mathbb{A} \cup \mathbb{B}$(Rectangles containing elements in $\mathbb{B}$ can be omited since their targets are elements in $Y$). This gives the isomorphism.
	
	Define a map
	\begin{equation}
		\begin{aligned}
			\Phi_A :Z &\longrightarrow Y
			\\
			\textbf{x} &\longrightarrow \sum_{\textbf{y} \in \textbf{S}(\mathbb{G}_B)} \sum_{\substack{r \in Rect^\circ (\textbf{x},\textbf{y})  \\ r \text{ contains exactly 1 }B \\ r \cap (\mathbb{X}\cup \mathbb{A})= \emptyset }}
			V_1^{O_1(r)}\cdots  V_{m+1}^{O_{m+1}(r)   } \textbf{y}
		\end{aligned}
	\end{equation}
	$\Phi_A$ is a chain map: As in Prop.\ref{1}, for $\psi \in \pi(\textbf{x},\textbf{y})$, let $N(\psi)$ be the number of ways we can decompose $\psi$ as the composition of 2 empty rectangles, we have  
	
	\begin{equation}\label{eq1}
		\left(\partial_Y \circ \Phi_A + \Phi_A \circ \partial_Z \right)(\textbf{x})= \sum_{\textbf{y} \in \textbf{S}(\mathbb{G}_B)} \sum_{\substack{\psi \in Rect^\circ (\textbf{x},\textbf{y})  \\ \psi \text{ contains exactly 1 }B \\ \psi \cap (\mathbb{X}\cup \mathbb{A})= \emptyset }}
		V_1^{O_1(\psi)}\cdots  V_{m+1}^{O_{m+1}(\psi)   } \textbf{y}
	\end{equation}
	
	As in the proof of Prop.\ref{1} , when $\text{x} \backslash \text{x}\cap\text{y}$ consists of 4 or 3 elements, $N(\psi)=2$, so they contribute nothing to the right of Equation (\ref{eq1}).
	However, this time there is no longer the $\textbf{x}=\textbf{z}$ case: an annulus containing $B$ must contains an $A$, which is not allowed. Therefore the right-hand side of      is 0 and $\Phi_A$ commutes with the differential(Since we work in mod2 coefficient).

	Similarly, we can define
	\begin{equation}
		\begin{aligned}
			\Phi_B :Y &\longrightarrow Z
			\\
			\textbf{y} &\longrightarrow \sum_{\textbf{z} \in \textbf{S}(\mathbb{G}_B)} \sum_{\substack{r \in Rect^\circ (\textbf{y},\textbf{z})  \\ r \text{ contains exactly 1 }A \\ r \cap (\mathbb{X}\cup \mathbb{B})= \emptyset }}
			V_1^{O_1(r)}\cdots  V_{m+1}^{O_{m+1}(r)   } \textbf{z}
			\end{aligned}
	\end{equation}
	and prove that $\Phi_B$ is a chain map as above.
	
	\begin{lemma}\label{lemma4}
		The composition $\Phi_A \circ \Phi_B$,$\Phi_B \circ \Phi_A$ are equal to multiplication by $V_a +V_b-V_c-V_d$.
	\end{lemma}
	\begin{proof}
		$\Phi_A \circ \Phi_B$ counts only the four annuli that contain $p$ on their boundary. The vertical annuli contribute $V_a$ and $V_b$, while the vertical ones contribute $-V_c$ and $-V_d$. The same is true for $\Phi_B \circ \Phi_A$.
	\end{proof}

	\textbf{Remark}. Since we are working in $\mathbb{F}$-coefficients, the signs appearing above are irrevalent. They are suggested by the sign conventions in the $\mathbb{Z}$-coefficient version.  

   \begin{lemma}\label{lemma5}
   	The mapping cone $C(\Phi_A)$ is chain isomorphic to $GC^-(\mathbb{G}_B)$, and the mapping cone $C(\Phi_B)$ of $\Phi_B$ is chain isomorphic to $GC^-(G_B)$.
   \end{lemma}

\begin{proof}
	By definition, the underlying module of $C(\Phi_B)$ is $Y \oplus X$, and the differential
	\begin{equation}
		\begin{aligned}
	Y \oplus X &\longrightarrow Y \oplus X  \\
	(y,x) &\longrightarrow (\partial_Y(y), \Phi_A(y)+ \partial_X(x))
		\end{aligned}
	\end{equation}
On the other hand, the differential $\partial$ of $GC^-(G_B)$ is 
\begin{equation}
	\begin{aligned}
		Y \oplus X &\longrightarrow Y \oplus X  \\
		(y,x) &\longrightarrow \partial_{\mathbb{G}_B}(y)+  \partial_{\mathbb{G}_B}(x)
	\end{aligned}
\end{equation}
where $\partial_{\mathbb{G}_B}(x)=\partial_X(x)$, $\partial_{\mathbb{G}_B}(y)= \partial_Y(y)+ \Phi_A(y)$, therefore
\begin{equation}
	\partial_{\mathbb{G}_B}(y)+  \partial_{\mathbb{G}_B}(x)= \partial_X(x)+\partial_Y(y)+ \Phi_A(y)=(\partial_Y(y), \Phi_A(y)+ \partial_X(x))
\end{equation}

The proof of the latter half statement is similar.
\end{proof}

Recall the following properties on mapping cone:
\begin{proposition}\label{prop1}
	(1) If $f:C \longrightarrow C'$ is an injective chain map, then there is a quasi-isomorphism $\phi: Cone(f) \longrightarrow \frac{C'}{f(C)}$;
	
	(2) Suppose that $f:C \longrightarrow C'$, $g:C' \longrightarrow C"$ are two chain maps, then there is a long exxact sequence:
	\begin{equation}
		\cdots \longrightarrow H(Cone(g)) \longrightarrow H(Cone(f)) \longrightarrow H(Cone(g \circ f)) \longrightarrow H(Cone(g)) \cdots
	\end{equation} 
\end{proposition}

\textbf{Proof of Theorem \ref{thm1}  }:  By Prop.\ref{prop1}(1), the mapping cone of multiplication by $V_a+V_b-V_c-V_d: X \longrightarrow X$ is a quasi-isomorphic to $\frac{X}{V_a+V_b-V_c-V_d}$. Then Equation \ref{a} follows from  Lemma  \ref{lemma3},\ref{lemma4}  \ref{lemma5} and Prop.\ref{prop1}(2). Choose a set $\{V_e\}$, one for each edge $e$, Equation (2)  follows by letting these $V_e=0$.

It remains to prove the part for $G_-$. Indeed, this follows directly from (1)(2) by the following observation: If we operate an RIV move on $G$, we get a spatial graph $G'$(equivalent to $G$). The two resolutions $G'_+$, $\mathcal{R}_{G'}$ of $G'$ are respectively $\mathcal{R}_G$ and $G^-$(See Fig.\ref{graph7}).

\begin{figure}[t]
	\def\svgwidth{\columnwidth}
	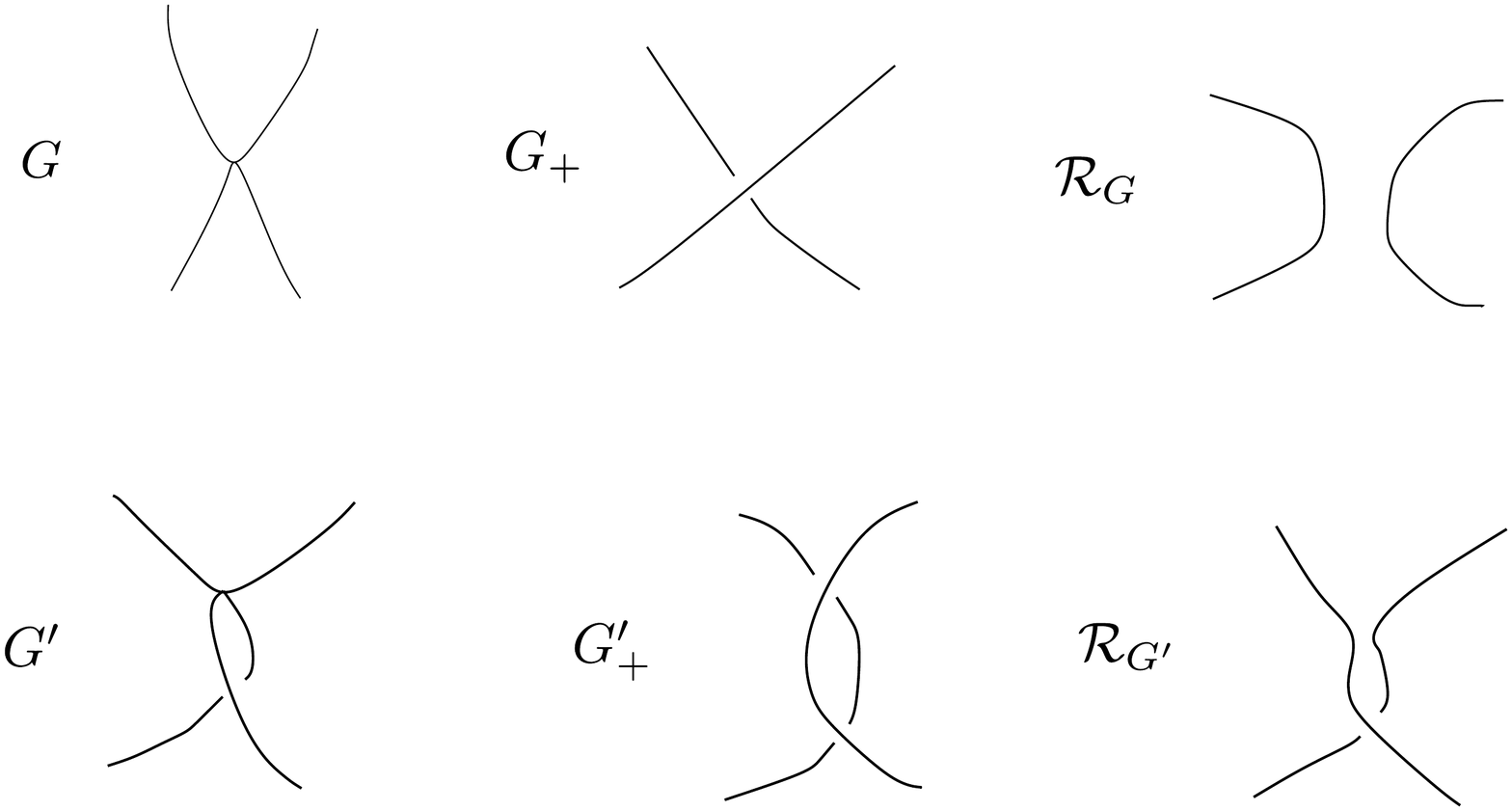
	\caption{Resolutions of $G$ and $G'$} 
	\label{graph7}
\end{figure}

\bibliography{gh}
\nocite{*}
\end{document}